\documentclass[11pt,reqno]{amsart}

\usepackage{amsmath,amsfonts,amsthm,amscd,amssymb,graphicx,mathrsfs}
\usepackage[utf8]{inputenc}
\usepackage{amsbsy}
\usepackage{graphicx}
\usepackage{subcaption}
\usepackage{hyperref}
\usepackage{bm}
\usepackage{tikz}
\usetikzlibrary{arrows}
\usepackage[text={6.5in,9in},centering,includefoot,foot=0.6in]{geometry}

\usepackage[all]{xy}
\usepackage[mathscr]{euscript}

\usepackage{amsmath,amsthm,amsfonts,amssymb,amscd}
\usepackage[all]{xy}
\usepackage{sseq}

\usepackage{todonotes}

\definecolor{skyblue}{rgb}{0.85,0.85,1}

\newtheorem{theorem}{Theorem}
\newtheorem{lemma}[theorem]{Lemma}

\newtheorem{cor}[theorem]{Corollary}

\theoremstyle{definition}
\newtheorem{rem}[theorem]{Remark}

\newtheorem{define}[theorem]{Definition}

\DeclareMathOperator{\dv}{div}

\DeclareMathOperator{\Hess}{Hess}

\DeclareMathOperator{\dom}{dom}
\DeclareMathOperator{\id}{id}
\DeclareMathOperator{\sgn}{sgn}

\newcommand{\bbR}{\mathbb{R}}

\newcommand{\cD}{\mathcal{D}}
\newcommand{\cE}{\mathcal{E}}

\newcommand{\cP}{\mathcal{P}}

\newcommand{\cU}{\mathcal{U}}

\newcommand{\pO}{{\partial \Omega}}

\newcommand{\p}{\partial}

\newcommand{\nor}{\p_{_\Sigma}}

\newcommand{\cEp}{\cE_{\scriptscriptstyle P}^s}
\newcommand{\cPp}{\cP_{\!\scriptscriptstyle P}^s}
\newcommand{\Tp}{T_{\scriptscriptstyle P}}
\renewcommand{\S}{S_{\scriptscriptstyle P}}
\newcommand{\Dtn}{\Lambda_{\scriptscriptstyle P}}
\newcommand{\DtnNu}{\Lambda_{\scriptscriptstyle P,\nu}}

\newcommand{\Fnu}{\mathscr{F}_{\scriptscriptstyle P,\nu}}
\newcommand{\Snu}{S_{\scriptscriptstyle P,\nu}}
\newcommand{\cH}{\mathcal{H}_{\scriptscriptstyle P,\nu}}

\begin{document}

\title[Stability of spectral partitions and the Dirichlet-to-Neumann map]
{Stability of spectral partitions and the Dirichlet-to-Neumann map}

\author{G. Berkolaiko}
\address{Department of
  Mathematics, Texas A\&M University, College Station, TX 77843-3368, USA}
\email{berko@math.tamu.edu}

\author{Y. Canzani}
\address{Department of Mathematics, University of North Carolina at Chapel Hill,
Phillips Hall, Chapel Hill, NC  27599, USA}
\email{canzani@email.unc.edu}

\author{G. Cox}
\address{Department of Mathematics and Statistics, Memorial University of Newfoundland, St. John's, NL A1C 5S7, Canada}
\email{gcox@mun.ca}

\author{J.L. Marzuola}
\address{Department of Mathematics, University of North Carolina at Chapel Hill,
Phillips Hall, Chapel Hill, NC  27599, USA}
\email{marzuola@math.unc.edu}

\begin{abstract}
The oscillation of a Laplacian eigenfunction gives a great deal of information about the manifold on which it is defined. This oscillation can be encoded in the \emph{nodal deficiency}, an important geometric quantity that is notoriously hard to compute, or even estimate. Here we compare two recently obtained formulas for the nodal deficiency, one in terms of an energy functional on the space of equipartitions of the manifold, and the other in terms of a two-sided Dirichlet-to-Neumann map defined on the nodal set. We relate these two approaches by giving an explicit formula for the Hessian of the equipartition energy in terms of the Dirichlet-to-Neumann map. This allows us to compute Hessian eigenfunctions, and hence directions of steepest descent, for the equipartition energy in terms of the corresponding Dirichlet-to-Neumann eigenfunctions. Our results do not assume bipartiteness, and hence are relevant to the study of spectral minimal partitions.
%
\end{abstract}

\maketitle

\section{Introduction}

Let $(M^n,g)$ be a compact Riemannian manifold, and denote the eigenvalues of the Laplace--Beltrami operator by $\lambda_1 < \lambda_2 \leq \cdots$, with corresponding eigenfunctions $\psi_1, \psi_2, \ldots$. For any eigenfunction $\psi_*$, Courant's nodal domain theorem says that its number of nodal domains, denoted $\nu(\psi_*)$, is bounded above by the minimal label of its eigenvalue, which is defined as $\ell(\psi_*) := \min\{k : \lambda_k = \lambda_* \}$. That is, any eigenfunction corresponding to the $k$-th eigenvalue has at most $k$ nodal domains.

Equivalently, the \underline{nodal deficiency}
\begin{equation}
	\delta(\psi_*) := \ell(\psi_*) - \nu(\psi_*)
\end{equation}
is nonnegative.
Despite almost a century of intensive study, this quantity is still not very well understood. Much attention has been paid to the so-called \underline{Courant sharp} eigenfunctions---those for which $\delta(\psi_*) = 0$. It is well known that there are only finitely many of these on any given domain \cite{P56}.
There are many examples of domains where one can exhaustively list the Courant sharp eigenfunctions; see, for instance, \cite{berard2016courant,berard2020courant,helffer2010spectral,helffer2016nodal,lena2015courant}, as well as the survey \cite{bonnaillie2015nodal} and references therein. However, these examples are all highly symmetric, and their analysis relies on explicit computation of the eigenfunctions and eigenvalues via separation of variables.

The first general formula for nodal deficiency on manifolds appeared in \cite{BKS12}, inspired by similar results for quantum graphs \cite{BBRS12}; see also \cite{BRS12} for the discrete graph setting. To describe this result, we require some definitions, which will be elaborated on in Section~\ref{sec:results}. We say that a $k$-partition $P = \{\Omega_j\}_{j=1}^k$ of $M$ is an \underline{equipartition} if
\begin{equation}
\label{lambda}
	\lambda_1(\Omega_1) = \cdots = \lambda_1(\Omega_k),
\end{equation}
where $\lambda_1(\Omega_j)$ denotes the first eigenvalue of the Dirichlet Laplacian on $\Omega_j$. For an equipartition $P$ we define $\lambda(P)$ to be the common value in \eqref{lambda}. The set of equipartitions near a given smooth (i.e. $C^\infty$) equipartition can be given the structure of a Hilbert manifold, on which $P \mapsto \lambda(P)$ is a smooth function. The nodal domains of a Laplacian eigenfunction are easily seen to form a bipartite equipartition; see Definition~\ref{def:bi}.  Conversely, it was shown in \cite{BKS12} that a smooth, bipartite equipartition $P$ is the nodal partition of an eigenfunction $\psi_*$ if and only if $P$ is
a critical point of $\lambda$.  Moreover, if the corresponding eigenvalue is simple\footnote{It is an immediate consequence of our main result that \eqref{nodal:BKS} also holds for non-simple eigenvalues. In this case the Hessian is degenerate, with nullity determined by the multiplicity of the eigenvalue, as in \eqref{nullmult}.},
then the Hessian of $\lambda$ at $P$ is non-degenerate, and its Morse index equals the nodal deficiency,
\begin{equation}
\label{nodal:BKS}
	n_-\big(\Hess \lambda(P)\big) = \delta(\psi_*).
\end{equation}
Here we recall that the \underline{Morse index} of a symmetric bilinear form, denoted $n_-$, is the maximal dimension of a subspace on which the form is negative definite, and the \underline{nullity}, $n_0$, is the dimension of the nullspace of the form. If the bilinear form corresponds to a self-adjoint operator, then the Morse index and nullity equal the number of negative and zero eigenvalues, respectively, counted with multiplicity.

It follows from \eqref{nodal:BKS} that smooth nodal partitions of Courant sharp eigenfunctions correspond to local minima of the equipartition energy. On the other hand, an earlier result in \cite{helffer2009nodal} showed that bipartite (globally) minimal partitions are precisely the nodal partitions of Courant sharp eigenfunctions, under a mild regularity assumption on the partition boundary. 
Combining this with \eqref{nodal:BKS}, we have that
\[
    \bigg\{\begin{tabular}{@{}c@{}}
        \text{smooth, bipartite} \\
        \text{local minima of $\lambda$}
    \end{tabular}\bigg\}
    \quad\Longleftrightarrow\quad
    \bigg\{\begin{tabular}{@{}c@{}}
        \text{smooth nodal partitions of } \\
        \text{Courant sharp eigenfunctions}
    \end{tabular}\bigg\}
    \quad\Longleftrightarrow\quad
    \bigg\{\begin{tabular}{@{}c@{}}
        \text{smooth, bipartite} \\
        \text{global minima of $\lambda$}
    \end{tabular}\bigg\}.
\]
That is, every smooth, bipartite local minimum of $\lambda$ is in fact a global minimum. A similar phenomenon was recently observed in the dispersion relations of periodic graphs \cite{BCCM}.

The second explicit formula for the nodal deficiency appeared in \cite{CJM2}; see also \cite{BCM19,helffer2021spectral}. To facilitate our comparison with \eqref{nodal:BKS}, we will state the result in a stronger form, which is due to \cite{BCHS}. 
Let  $P=\{\Omega_j\}$ be the nodal partition of an eigenfunction $\psi_*$, with energy $\lambda(P) = \lambda_*$.
We first introduce the two-sided Dirichlet-to-Neumann map $\Dtn$ associated to the eigenvalue $\lambda_*$. This is an unbounded, self-adjoint operator, with domain dense in
\begin{equation}
\label{def:S}
	\S := \left\{ f \in L^2(\Sigma) : \int_{\pO_j} f \frac{\p \psi_j}{\p \nu_j} = 0 \text{ for all } j \right\},
\end{equation}
where $\psi_j$ denotes the ground state for the Dirichlet Laplacian on $\Omega_j$,  $\nu_j$ is the outward unit normal, and $\Sigma:=\cup_j\partial\Omega_j$ is the nodal set. A precise definition will be given in Section \ref{sec:DtN}; for now we just mention that
\begin{equation}
\label{LambdaS}
	\Dtn f = \Pi_{\S}\big(\nor u\big)
\end{equation}
for sufficiently smooth $f \in  \S$, where $\nor u$ is a function on $\Sigma$ given by $\nor u\big|_{\pO_i \cap \pO_j} : = \frac{\p u_i}{\p \nu_i} + \frac{\p u_j}{\p \nu_j}$ for $i \neq j$,
 $u_j$ is any solution to the boundary value problem
$\Delta u_j + \lambda_* u_j = 0$ in  $\Omega_j$ with $u_j\big|_{\pO_j} = f,$
and $\Pi_{\S}$ is the $L^2(\Sigma)$-orthogonal projection onto $\S$. Since $\nu_i = -\nu_j$ on $\pO_i \cap \pO_j$, the function $\nor u$  measures the mismatch in normal derivatives across the nodal set $\Sigma$.

The result from \cite{BCHS} can now be stated as follows: If $P=\{\Omega_j\}$ is the nodal partition of an eigenfunction $\psi_*$, with energy $\lambda(P) = \lambda_*$, then
\begin{align}
\label{nodal:CJM}
\begin{split}
	n_-(\Dtn) = \delta(\psi_*), \qquad
	n_0(\Dtn) = n_0 (\Delta + \lambda_*) -1.
\end{split}
\end{align}

Comparing the formulas \eqref{nodal:BKS} and \eqref{nodal:CJM} for the nodal deficiency, we see that
\begin{align}
\label{nodal:equal}
\begin{split}
	n_-\big(\Hess \lambda(P)\big) &= n_-(\Dtn). \\
\end{split}
\end{align}
The goal of this paper is to explain \emph{why} this equality holds. We achieve this by giving an explicit relationship between $\Hess \lambda(P)$ and $\Dtn$. Namely, in Theorem \ref{thm:Hess1}, we prove that the bilinear form $\Hess \lambda(P)$ generates a self-adjoint operator that is unitarily equivalent to $\Dtn$. We only consider smooth\footnote{The non-smooth case is more involved and will be treated elsewhere; see the discussion at the end of Section~\ref{sec:results}.} equipartitions, but we do not require them to be associated to eigenfunctions that have simple eigenvalues. In particular, our results imply that \eqref{nodal:BKS} remains valid for the nodal partition of an eigenfunction with non-simple eigenvalue, and also give the equality
\begin{equation}
\label{nullmult}
	n_0\big(\Hess \lambda(P)\big) =  n_0 (\Delta + \lambda_*) -1
\end{equation}
for the nullity of the Hessian. 

Moreover, our results also apply to non-bipartite partitions, with a suitable modification of $\Dtn$. (Note that smooth, non-bipartite partitions can only exist on multiply connected domains, such as the torus.) This gives a powerful new tool in the study of spectral minimal partitions, since our analysis provides explicit formulas relating the eigenfunctions of the Dirichlet-to-Neumann map to the directions of steepest descent for the function $\lambda$. We illustrate this point with an example in Section \ref{sec:example}, where we see a compelling geometric connection between the eigenfunctions of $\Dtn$ and the conjectured minimal 3-partition of the square.

\section{Statement of results}
\label{sec:results}
To illustrate the relationship between the Hessian and the Dirichlet-to-Neumann map with minimal technicalities, we assume that $\p M = \varnothing$, and only deal with generic partitions, as defined below. The case of non-generic partitions will be treated in a future work.

\begin{define}
\label{def:generic}
$P = \{\Omega_j\}$ is said to be a \underline{generic $k$-partition of $M$} if $\Omega_1, \ldots, \Omega_k$ are nonempty, disjoint, open, connected subsets of $M$ such that:
\begin{enumerate}
\item each $\Omega_j$ is a smooth manifold with boundary,
\item $M=\Omega_1 \cup \cdots \cup \Omega_k \cup \Sigma$, where $\Sigma := \bigcup_{j=1}^k \pO_j$,
\item  for each $j$, the normal derivative of the ground state $\psi_j$ for the Laplacian on  $\Omega_j$ is nowhere vanishing on $\p\Omega_j$.
\end{enumerate}
\end{define}

These are generic properties in the sense that for a residual set of Riemannian metrics $g$ on $M$, every eigenfunction of the Laplace--Beltrami operator $\Delta_g$ generates a nodal partition satisfying Definition~\ref{def:generic}; see \cite{U76} for details. Generic partitions are by definition exhaustive. The condition (1) is stronger than requiring the set $\Sigma$ to be a smoothly embedded hypersurface. A simple example is when $M$ is a 2-torus and $\Sigma \subset M$ is a smooth, non-separating loop, so that $\Omega := M \setminus \Sigma$ is connected. In this case the topological boundary $\pO = \Sigma$ is smooth, but $\Omega$ lies on both sides of $\Sigma$, and hence is not a manifold with boundary.

We also recall the notion of a bipartite partition, emphasizing that generic partitions are not required to satisfy this condition. First, we declare that two subdomains $\Omega_i$ and $\Omega_j$, with $i \neq j$, are \underline{neighbors} if $\pO_i \cap \pO_j \neq \varnothing$.

\begin{define}
\label{def:bi}
A generic partition $P = \{\Omega_j\}$ is said to be \underline{bipartite} if there exists a function $\eta \colon \{\Omega_j\} \to \{\pm1\}$ such that $\eta(\Omega_i) = -\eta(\Omega_j)$ whenever $\Omega_i$ and $\Omega_j$ are neighbors.
\end{define}

The nodal partition of an eigenfunction $\psi$ is always bipartite\,---\,to prove this one simply defines $\eta(\Omega_j) = \sgn \big(\psi\big|_{\Omega_j}\big)$.

Before calculating the Hessian of $\lambda$, we need to know the manifold structure of the space of equipartitions. Let $P = \{\Omega_j\}$ be a generic $k$-equipartition, and fix a number $s>(n+3)/2$. One may endow the space of $k$-partitions near $P$ with a smooth structure in which nearby partitions are realized as perturbations of $P$, obtained by deforming $\Sigma$ in the normal direction with the deformation parameterized by a function in $H^s(\Sigma)$; see Section \ref{sec:manifold} for details.
With this structure in place, the set $\cEp$ of equipartitions that are close to $P$ is a smooth Hilbert manifold, and it is shown in \cite[Proposition~8]{BKS12}
that the function 
$
\lambda \colon \cEp \to \bbR
$ is smooth. In addition, \cite[Theorem~9]{BKS12} characterizes the critical points of $\lambda$, concluding that $D\lambda(P)=0$ if and only if there exist nonzero real numbers $a_1, \ldots, a_k$ such that
\begin{equation}
\label{eq:normals}
    \left| a_i \frac{\p \psi_i}{\p \nu_i} \right| = \left| a_j \frac{\p \psi_j}{\p \nu_j} \right| \;\; \text{on}\;\; \pO_i \cap \pO_j
\end{equation}
for all $i,j$.
We assume that the $a_j$ are normalized to have $a_1^2 + \cdots + a_k^2 = 1$.
This condition, together with \eqref{eq:normals}, determines each $a_j$ up to a sign. If $P$ is bipartite, it is natural to fix the signs by choosing $\sgn a_j = \eta(\Omega_j)$ for each $j$. In this case the function $\psi$ defined by $\psi\big|_{\Omega_j} = a_j \psi_j$ belongs to  $H^2(M)$, and hence is a global Laplacian eigenfunction, which means $P$ is a nodal partition. However, we emphasize that in general we do not require $P$ to be bipartite.

Assuming $P$ is a critical partition, we choose $\{a_j\}$ as above and define a weight function
\begin{equation}
\label{rhodef}
	\rho\colon \Sigma \to \mathbb R, \qquad \rho\big|_{\pO_j} := \left| a_j \frac{\p \psi_j}{\p \nu_j} \right|.
\end{equation}
The criticality condition \eqref{eq:normals} ensures that $\rho$ is well defined. We then define weighted spaces
\begin{equation}
	L^2_\rho(\Sigma) := \left\{ \phi : \rho\phi \in L^2(\Sigma) \right\}, \qquad \left< \phi_1,\phi_2\right>_{L^2_\rho(\Sigma)} := \left< \rho\phi_1, \rho\phi_2 \right>_{L^2(\Sigma)}
\end{equation}
and
\begin{equation}
\label{Hsrho}
	H^s_\rho(\Sigma) := \left\{ \phi : \rho\phi \in H^s(\Sigma) \right\}, \qquad \left< \phi_1,\phi_2\right>_{H^s_\rho(\Sigma)} := \left< \rho\phi_1, \rho\phi_2 \right>_{H^s(\Sigma)}.
\end{equation}
The genericity assumption on $P$ implies that both $\rho$ and $\rho^{-1}$ are smooth and bounded away from zero, so the weighted and unweighted inner products are equivalent; see Remark~\ref{rem:weight} for further discussion.

Finally, we let $\nu$ be a smooth unit normal vector field along $\Sigma$. As explained in Section \ref{sec:manifold}, this allows us to parameterize $\cEp$ using functions, rather than vector fields, on $\Sigma$. We then introduce a modified version of the two-sided Dirichlet-to-Neumann map, denoted $\DtnNu$ (see Section~\ref{sec:DtN} for a precise definition). 
While the operator $\DtnNu$ depends on the choice of $\nu$, we will see below that its index does not.

The main result of this paper describes the relationship between the modified Dirichlet-to-Neumann map (a self-adjoint operator), the Hessian (a closable bilinear form), and the self-adjoint operator generated by the closure of the Hessian. In what follows, we write\footnote{Throughout the paper, all integrals are with respect to the Riemannian volume measure on $M$, or the induced surface measure on $\Sigma$; we do not indicate the measure explicitly since it will always be clear from the context.} 
\begin{equation}
\label{def:SH}
	\Fnu := \Bigg\{\phi \in L^2_\rho(\Sigma) : \int_{\pO_j} (\nu\cdot\nu_j) \phi\left(\frac{\p \psi_j}{\p \nu_j}\right)^2 = 0 \text{ for all } j \Bigg\}.
\end{equation}
We will see below that if $P$ is a critical partition, then $H^s_\rho(\Sigma) \cap \Fnu$ coincides with $\Tp \cEp$, the tangent space at $P$ to the manifold $\cEp$ of nearby equipartitions.

\begin{theorem}
\label{thm:Hess1}
Fix $s>(n+3)/2$ and let $P$ be a generic critical equipartition for $\lambda\colon\cEp \to \mathbb R$. If $\nu$ is a smooth unit normal vector field along $\Sigma$, then
  \begin{equation}
\label{Hess1}
	\Hess \lambda(P)(\phi_1\nu,\phi_2\nu) = 2\left< \DtnNu(\rho\phi_1), \rho\phi_2 \right>_{L^2(\Sigma)}
\end{equation}
for all $\phi_1, \phi_2 \in H^s_\rho(\Sigma)\cap \Fnu$. The bilinear form  $h(\phi_1,\phi_2) := \Hess \lambda(P)(\phi_1\nu,\phi_2\nu)$, with $\dom(h) = H^s_\rho(\Sigma)\cap \Fnu$, is semibounded and closable on $\Fnu$,
and therefore generates a self-adjoint operator $\cH$, which is given by
\begin{equation}
\label{eq:H}
	\cH (\phi) = 2  \rho^{-1} \DtnNu (\rho \phi)
\end{equation}
and has domain
\begin{equation}
\label{eq:domain}
H^1_\rho(\Sigma) \cap \Fnu \subseteq \dom(\cH) \subseteq H^{1/2}_\rho(\Sigma) \cap \Fnu.
\end{equation}
\end{theorem}

\begin{rem}
\label{rem:weight}
The weight $\rho$ may appear to be unnecessary, since the $L^2$ and $L^2_\rho$ norms are equivalent, and similarly for $H^s$ and $H^s_\rho$, so it does not affect the closability of $h$. However, it is important for two reasons:
\begin{enumerate}
    \item It ensures that $\cH$ is unitarily equivalent to $\DtnNu$, and not merely congruent (Corollary~\ref{cor:Morse}).
    \item In the non-generic case, where the nodal lines are allowed to intersect, the weight $\rho$ will vanish at these points. When this happens the norms are no longer equivalent, and one must use the weighted norm to obtain a closable bilinear form.
\end{enumerate}
Therefore, we describe the form domain in terms of the weighted space $H^s_\rho$, in order to be consistent with future work where this distinction will be crucial \cite{BCCKM}.
\end{rem}

We assume for the rest of this section that $s$ and $\nu$ have been fixed. Since multiplication by $\rho$ gives an isometric isomorphism from $L^2_\rho(\Sigma)$ to $L^2(\Sigma)$, 
we get the following.

\begin{cor}
\label{cor:Morse}
If $P$ is a generic critical equipartition,
then $\cH$ is unitarily equivalent to $2\DtnNu$.
\end{cor}

This allows us to compute eigenvalues and eigenfunctions of $\cH$ using $\DtnNu$. However, we are ultimately interested in $\Hess \lambda(P)$, rather than its closure (or the corresponding self-adjoint operator $\cH$), which is defined on a strictly larger domain. The domain inclusion implies
\begin{equation}
\label{Morseleq}
	n_-\big(\Hess \lambda(P)\big) \leq n_-(\cH) = n_-(\DtnNu),
\end{equation}
and similarly for the nullity $n_0$. For a generic partition we prove that this is actually an equality.

\begin{theorem}
\label{thm:smooth}
If $P$ is a generic critical equipartition, then
\begin{equation}
\label{n-all}
	n_-\big(\Hess \lambda(P)\big) = n_-(\cH) = n_-(\DtnNu),
\qquad 
	n_0\big(\Hess \lambda(P)\big) = n_0(\cH) = n_0(\DtnNu).
\end{equation}
\end{theorem}
This is essentially a regularity statement\,---\,we show that the eigenfunctions of $\cH$ are smooth, and hence contained in $H^s(\Sigma)$ regardless of the choice of $s$.

It was shown in \cite{BCHS} that the Morse index of $\DtnNu$ equals the \emph{defect} of the partition, a quantity that generalizes the nodal deficiency in the non-bipartite case. Combining this with Theorem~\ref{thm:smooth} therefore extends the results of \cite{BKS12}, which only treated nodal (and hence bipartite) partitions.

It is clear from \eqref{n-all} that the index and nullity of $\DtnNu$ (and also of $\cH$) do not depend on the choice of $\nu$. We will see that different choices of $\nu$ lead to unitarily equivalent Dirichlet-to-Neumann operators. In the bipartite case it follows that $\DtnNu$ is unitarily equivalent to $\Dtn$ for \emph{any} choice of $\nu$; see Remarks~\ref{rem:Lambda} and~\ref{rem:nu}.


\begin{cor}
\label{cor:bipartite}
If $P$ is a generic bipartite critical equipartition, then
\begin{equation}
	n_-\big(\Hess \lambda(P)\big) = n_-(\Dtn), \qquad n_0\big(\Hess \lambda(P)\big) = n_0(\Dtn).
\end{equation}
\end{cor}

This is the desired equality \eqref{nodal:equal} for generic nodal partitions. However, the significance of our results goes far beyond establishing this equality.
In particular, it gives us a means of finding eigenfunctions of the Hessian in terms of the two-sided Dirichlet-to-Neumann map. Indeed, we see that $\phi \in H^1_\rho(\Sigma)$ is an eigenfunction of $\cH$ if and only if
\begin{equation}
\label{eigenfunction}
	\rho \phi \in H^1(\Sigma)
\end{equation}
is an eigenfunction of $\DtnNu$. Therefore, we can find eigenfunctions of $\cH$ by computing $\DtnNu$ eigenfunctions and then dividing by the weight $\rho$, which is nonvanishing by our genericity assumption. An example of this procedure is given in Section~\ref{sec:example}.

We expect these results will be useful in the study of spectral minimal partitions, which are partitions that minimize the quantity
\[
	\max_{1 \leq j \leq k} \lambda_1(\Omega_j).
\]
It is known that such minimal partitions always exist, are equipartitions, and satisfy certain regularity properties; see \cite{helffer2009nodal} and references therein. However, they do not necessarily satisfy the genericity conditions in Definition \ref{def:generic}. In particular, the set $\Sigma$ may contain self intersections, in which case it is not smooth.

Generalizing the above results to this case is significantly more involved, and will be addressed in a future work \cite{BCCKM}. Some difficulties of dealing with non-generic partitions were explored in a recent series of papers
on quantum graphs \cite{HK21,HKMP21,KKLM21}. Here we mention some of the difficulties that arise on manifolds. To begin with, the structure of the space of partitions becomes more complicated when self-intersections are allowed. Moreover, the weight function $\rho$ will vanish at the points of intersection. Therefore, if $f \in H^1(\Sigma)$ is an eigenfunction for $\DtnNu$, it will still be the case that $\rho^{-1} f \in H^1_\rho(\Sigma)$ is an eigenfunction for $\cH$, but we can no longer guarantee that $\rho^{-1} f$ is smooth, which means it may not be contained in the domain of $\Hess \lambda(P)$ (i.e. the tangent space to the manifold of equipartitions). As a result, the inequality \eqref{Morseleq} may be strict. This suggests that there are ``deformations" of $P$ that decrease the energy $\lambda$ but are not smooth, e.g. they change the topology of the nodal set.

\subsection*{Outline}
In Section \ref{sec:prelim} we review some fundamental definitions and constructions from \cite{BCHS,BCM19,BKS12} and \cite{CJM2}, which form the basis for our analysis. In Section \ref{sec:var} we compute the Hessian of $\lambda$, establishing \eqref{Hess1}. In Section \ref{sec:close} we describe the closure of the Hessian, which yields Theorem \ref{thm:Hess1}, and then prove all of the corollaries. 
Finally, in Section \ref{sec:example} we illustrate our main results and formulas with an example.

\subsection*{Acknowledgments}
The authors thank Ram Band, Sebastian Egger, Bernard Helffer, Peter Kuchment, and Mikael Persson Sundqvist for helpful comments and discussions.  G.B. acknowledges the support of NSF Grant DMS-1815075.  Y.C. was supported by the Alfred P. Sloan Foundation and NSF CAREER Grant DMS-2045494 and DMS-1900519. G.C. acknowledges the support of NSERC grant RGPIN-2017-04259.  J.L.M.
acknowledges support from the NSF through NSF CAREER Grant
DMS-1352353 and NSF grant DMS-1909035.  The authors are grateful to the AIM SQuaRE program for hosting them and supporting the initiation of this project.


\section{Preliminaries}
\label{sec:prelim}
Before proving our main results, we review the definitions of the objects that appear in the statements of those results, namely the manifold of equipartitions and the two-sided Dirichlet-to-Neumann map.

\subsection{The manifold of equipartitions}
\label{sec:manifold}
We first describe the set $\cPp$ of $k$-partitions close to $P$, and then the subset $\cEp \subset \cPp$ of equipartions, which is a submanifold of codimension $k-1$.

Assuming that $P$ is a generic $k$-partition, with nodal set $\Sigma$, we let $H^s(\Sigma)$ denote the Sobolev space of $H^s$ functions on $\Sigma$, and similarly for $H^s(M)$.
We also let $\cD^s(M)$ denote the set of $H^s$ diffeomorphisms of $M$. It is natural to parameterize $\cPp$ using vector fields defined along $\Sigma$. We find it more convenient to work with functions, however, so we fix\footnote{The smooth structure does not depend on the choice of unit normal, so we may assume that this is the same $\nu$ that appears in the statement of Theorem~\ref{thm:Hess1}.} a smooth unit normal vector field $\nu$ along $\Sigma$, and extend it arbitrarily to a smooth vector field $\tilde\nu$ on all of $M$. (The extension $\tilde\nu$ is allowed to vanish away from $\Sigma$, so there are no topological obstructions to its existence.)

We next fix a value of $s > (n+3)/2$ and choose a bounded extension operator $E^s \colon H^s(\Sigma) \to H^{s+1/2}(M)$. For any $\phi \in H^s(\Sigma)$ we let $\varphi_{\scriptscriptstyle \phi}$ denote the flow along the vector field $(E^s \phi)\tilde\nu$, evaluated at time $t=1$. Our choice of $s$ guarantees that $(E^s \phi)\tilde\nu$ is of class $H^{s+1/2}$ with $s+1/2 > n/2 + 2$, so \cite[Theorem~3.1]{EM70} implies $\varphi_{\scriptscriptstyle \phi} \in \cD^{s+1/2}(M)$. We then define
\begin{equation}
	\cPp = \big\{ \varphi_{\scriptscriptstyle \phi}(P) : \phi \in \cU \big\},
\end{equation}
where $\cU \subset H^s(\Sigma)$ is a neighborhood of zero. For $\cU$ sufficiently small the map $\phi \mapsto \varphi_{\scriptscriptstyle \phi}(P)$ is injective, and hence gives a bijection from $\cU$ onto $\cPp$. This gives $\cPp$ the structure of a smooth Hilbert manifold, and the tangent space at $P$ can be identified with $H^s(\Sigma)$.

\begin{rem}
The space $\cPp$ is automatically a smooth  manifold because it can be covered by a single coordinate chart, so there are no overlap/compatibility conditions to check.
This is no longer true if one considers the larger space
$
	\big\{ \varphi(P) : \varphi \in \cD^{s+1/2} \big\}
$
of all partitions that are $H^s$-diffeomorphic (but not necessarily close) to $P$, but this distinction is irrelevant for the current paper as we are only interested in local computations.
\end{rem}


We now define the subset $\cEp$ of equipartitions by
\begin{equation}
	\cEp = \big\{ \tilde P = \{\tilde\Omega_j\} \in \cPp :\; \lambda_1(\tilde\Omega_1) = \cdots = \lambda_1(\tilde\Omega_k) \big\}.
\end{equation}
Defining a map $\Xi \colon \cPp \to \bbR^k$ by $\Xi(\tilde P) = \big(  \lambda_1(\tilde\Omega_1), \ldots, \lambda_1(\tilde\Omega_k) \big)$, 
we have that $\cEp \subset \cPp$ is the preimage of the diagonal in $\bbR^k$, and it follows from a transversality argument, given in \cite[Section~3.1]{BKS12}, that it is a smoothly embedded submanifold of codimension $k-1$.

Recalling that $\Tp \cPp$ can be identified with $H^s(\Sigma)$, or equivalently $H^s_\rho(\Sigma)$, the tangent space to $\cEp$ will consist of the variations that preserve the equipartition condition, meaning the first variation of the ground state energy on each $\Omega_j$ is the same. By Hadamard's formula, this is equivalent to requiring that the integrals
\begin{equation}
\label{eq:Ham}
    \int_{\pO_j} (\phi\nu) \cdot \nu_j \left(\frac{\p \psi_j}{\p \nu_j}\right)^2
\end{equation}
coincide for all $j=1, \ldots, k$. The tangent space to $\cEp$ at $P$ can thus be described as
\begin{equation}\label{e:tangent}
	\Tp \cEp = \left\{\phi \in H^s_\rho(\Sigma) : \int_{\pO_1} \chi_1 \phi\left(\frac{\p \psi_1}{\p \nu_1}\right)^2 = \cdots
	= \int_{\pO_k} \chi_k \phi \left(\frac{\p \psi_k}{\p \nu_k}\right)^2 \right\},
\end{equation}
where we have defined
\begin{equation}\label{eqn:chis}
    \chi_j \colon \pO_j \to \{\pm1\},\qquad 	\chi_j = \nu \cdot \nu_j
\end{equation}
for each $j$. If $P$ is a generic \emph{critical} equipartition, then all of the integrals in \eqref{eq:Ham} will vanish, and we obtain
\begin{equation}
    \Tp \cEp = H^s_\rho(\Sigma) \cap \Fnu,
\end{equation}
where $\Fnu$ is defined in \eqref{def:SH}.



On each connected component of $\pO_j$ we will have either $\chi_j = 1$ or $\chi_j = -1$, but it is possible that both signs occur on different components of the boundary\,---\,if $P$ is non-bipartite this is inevitable. Some different choices of $\nu$, and the resulting $\chi_j$, are shown for a 3-partition of the circle in Figure~\ref{fig:nu}. 

\begin{lemma}
\label{lemma:bi}
A generic partition $P$ is bipartite if and only if there exists a choice of $\nu$ for which every $\chi_j$ is constant.
\end{lemma}

\begin{proof}
If $P$ is bipartite, we choose $\nu$ so that $\nu\big|_{\pO_j} = \eta(\Omega_j) \nu_j$ for each $j$. Definition~\ref{def:bi} guarantees this is well defined: if $\Omega_i$ and $\Omega_j$ are neighbors, then $\eta(\Omega_i) \nu_i = \eta(\Omega_j) \nu_j$, since $\eta(\Omega_i) = - \eta(\Omega_j)$ and $\nu_i = - \nu_j$ on $\pO_i \cap \pO_j$. With this choice of $\nu$ we have that $\chi_j = \eta(\Omega_j)$ is constant.

Conversely, if each $\chi_j$ is constant, we define $\eta(\Omega_j) = \chi_j$. To see that this satisfies Definition~\ref{def:bi}, we simply observe that if $\Omega_i$ and $\Omega_j$ are neighbors, then $\nu_i = - \nu_j$ on $\pO_i \cap \pO_j$, and hence $\chi_i = - \chi_j$.
\end{proof}

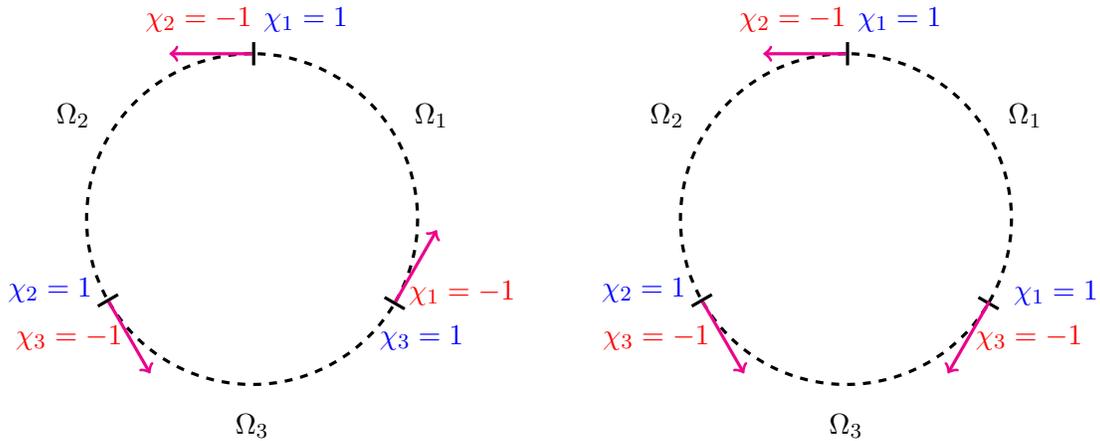
\begin{figure}
\begin{tikzpicture}[ scale=1.1]
	\draw[-|,very thick,dashed] (0,2) arc[radius=2, start angle=90, end angle=210];
	\node at ({2.5*cos(30)},{2.5*sin(30)}) {$\Omega_1$};
	\draw[-|,very thick,dashed] ({2*cos(210)},{2*sin(210)}) arc[radius=2, start angle=210, end angle=330];
	\node at ({2.5*cos(150)},{2.5*sin(150)}) {$\Omega_2$};
	\draw[-|,very thick,dashed] ({2*cos(330)},{2*sin(330)}) arc[radius=2, start angle=330, end angle=450];
	\node at ({2.5*cos(270)},{2.5*sin(270)}) {$\Omega_3$};
	
	\draw[->,very thick,magenta] (0,2) -- ++(180:1); 
	\draw[->,very thick,magenta] ({2*cos(210)},{2*sin(210)}) -- ++(300:1); 
	\draw[->,very thick,magenta] ({2*cos(330)},{2*sin(330)}) -- ++(60:1); 
	
	\node[blue] at ({2.5*cos(75)},{2.5*sin(75)}) {$\chi_1 = 1$};
	\node[red] at ({2.5*cos(105)},{2.5*sin(105)}) {$\chi_2 = -1$};
	\node[blue] at ({2.6*cos(200)},{2.5*sin(200)}) {$\chi_2 = 1$};
	\node[red] at ({2.7*cos(215)},{2.5*sin(215)}) {$\chi_3 = -1$};
	\node[blue] at ({2.5*cos(325)},{2.5*sin(325)}) {$\chi_3 = 1$};
	\node[red] at ({2.7*cos(340)},{2.6*sin(340)}) {$\chi_1 = -1$};
\end{tikzpicture}
\qquad
\begin{tikzpicture}[scale=1.1]
	\draw[-|,very thick,dashed] (0,2) arc[radius=2, start angle=90, end angle=210];
	\node at ({2.5*cos(30)},{2.5*sin(30)}) {$\Omega_1$};
	\draw[-|,very thick,dashed] ({2*cos(210)},{2*sin(210)}) arc[radius=2, start angle=210, end angle=330];
	\node at ({2.5*cos(150)},{2.5*sin(150)}) {$\Omega_2$};
	\draw[-|,very thick,dashed] ({2*cos(330)},{2*sin(330)}) arc[radius=2, start angle=330, end angle=450];
	\node at ({2.5*cos(270)},{2.5*sin(270)}) {$\Omega_3$};
	
	\draw[->,very thick,magenta] (0,2) -- ++(180:1); 
	\draw[->,very thick,magenta] ({2*cos(210)},{2*sin(210)}) -- ++(300:1); 
	\draw[->,very thick,magenta] ({2*cos(330)},{2*sin(330)}) -- ++(240:1); 
	
	\node[blue] at ({2.5*cos(75)},{2.5*sin(75)}) {$\chi_1 = 1$};
	\node[red] at ({2.5*cos(105)},{2.5*sin(105)}) {$\chi_2 = -1$};
	\node[blue] at ({2.6*cos(200)},{2.5*sin(200)}) {$\chi_2 = 1$};
	\node[red] at ({2.8*cos(215)},{2.5*sin(215)}) {$\chi_3 = -1$};
	\node[red] at ({2.7*cos(325)},{2.5*sin(325)}) {$\chi_3 = -1$};
	\node[blue] at ({2.7*cos(340)},{2.6*sin(340)}) {$\chi_1 = 1$};
\end{tikzpicture}
\caption{Two different choices of unit normal $\nu$, and the resulting $\chi_j$, for a 3-partition of the circle. In the left figure none of the $\chi_j$ are constant, i.e. each assumes both values $\pm1$. In the right figure $\chi_1 \equiv 1$ and $\chi_3 \equiv -1$ are constant but $\chi_2$ changes sign. (In this example $M=S^1$ is one-dimensional and $\Sigma$ consists of three points.)}
\label{fig:nu}
\end{figure}

\subsection{The two-sided Dirichlet-to-Neumann map}
\label{sec:DtN}

We now recall the definition of the two-sided Dirichlet-to-Neumann map $\DtnNu$, with $\Dtn$ in \eqref{LambdaS} appearing as a special case. The definition is complicated by the fact that $\lambda_*=\lambda(P)$ is in the Dirichlet spectrum on each nodal domain; in \cite{CJM2} the Dirichlet-to-Neumann map was defined for $\Delta + (\lambda_* + \varepsilon)$ precisely to avoid this difficulty.

However, there are two advantages to working with $\varepsilon=0$ directly: 1) it gives to a stronger result in the case of a multiple eigenvalue, as recently observed in \cite{BCHS}; and 2) it is precisely the operator that shows up in Theorem \ref{thm:Hess1} when we compute the Hessian of $\lambda$.

Throughout this section we assume that $\{\Omega_j\}$ is a generic equipartition with energy $\lambda(P) = \lambda_*$ and we fix a smooth unit normal vector field $\nu$ along $\Sigma$. With $\{\chi_j\}$ as in \eqref{eqn:chis}, we start by defining the closed subspace
\begin{equation}
\label{Schidef}
	\Snu := \Bigg\{f \in L^2(\Sigma) : \int_{\pO_j} \chi_j f \frac{\p\psi_j}{\p\nu_j} = 0 \text{ for all } j \Bigg\}
\end{equation}
of $L^2(\Sigma)$.  We will obtain $\DtnNu$ as the self-adjoint operator corresponding to a closed, semibounded bilinear form on a dense subspace of $\Snu$.

If $f \in H^{1/2}(\Sigma) \cap \Snu$, the boundary value problem 
\begin{equation}
\label{BVP}
	\Delta u_j + \lambda_* u_j = 0 \ \text{ in } \Omega_j, \qquad\qquad u_j\big|_{\pO_j} = \chi_j f,
\end{equation}
has a solution for each $j$; see, for instance \cite[Theorem~4.10]{M00}. Moreover, there exists a unique solution, which we denote $u_j^f$, satisfying the additional constraint $\int_{\Omega_j} u_j^f \psi_j = 0$. We then define the bilinear form
\begin{equation}
\label{adef}
	a(f,g) = \sum_{j=1}^k \int_{\Omega_j} \big( \nabla u_j^f \cdot \nabla u_j^g - \lambda_* u_j^f u_j^g \big),
\end{equation}
with domain $H^{1/2}(\Sigma) \cap \Snu$ dense in $\Snu$. It is easily shown (see \cite{AM12,BCHS}) that 
there are constants $C,c > 0$ and $m \in \bbR$ such that
\begin{equation}
\label{a:bound}
	|a(f,g)| \leq C \|f\|_{H^{1/2}(\Sigma)}  \|g\|_{H^{1/2}(\Sigma)} 
\end{equation}
and
\begin{equation}
\label{a:coer}
	a[f] \geq c \|f\|_{H^{1/2}(\Sigma)}^2 +m \|f\|^2_{L^2(\Sigma)}
\end{equation}
for all $f,g \in \dom(a)$. This means $a$ is closed and semibounded, so it generates a self-adjoint operator, which we denote $\DtnNu$, with $\dom(\DtnNu) \subseteq H^{1/2}(\Sigma) \cap \Snu$.

To characterize the domain of $\DtnNu$, we define the two-sided normal derivative distribution $\nor u^f \in H^{-1/2}(\Sigma)$ by
\begin{equation}
    \nor u^f := E_1\bigg(\chi_1\frac{\p u_1^f}{\p\nu_1}\bigg) + \cdots + E_k\bigg(\chi_k \frac{\p u_k^f}{\p\nu_k}\bigg),
\end{equation}
where $\p u_j^f/\p\nu_j \in H^{-1/2}(\pO_j)$ and 
\begin{equation}\label{Ej}
E_j \colon H^{-1/2}(\pO_j) \to H^{-1/2}(\Sigma)
\end{equation}
denotes the extension by zero. If $u^f$ is sufficiently smooth we will have $\p u_j^f/\p\nu_j \in L^2(\pO_j)$ for each $j$, in which case $\nor u^f$ is a function, given by
\[
\nor u^f\big|_{\pO_i \cap \pO_j} = \chi_i\frac{\p u^f_i}{\p \nu_i} + \chi_j\frac{\p u^f_j}{\p \nu_j}
\]
for $i \neq j$.

It is easily seen that
\begin{equation}
\label{Lambdadomain}
	\dom(\DtnNu) = \big\{f \in H^{1/2}(\Sigma) \cap \Snu : \nor u^f \in L^2(\Sigma) \big\},
\end{equation}
and for any $f \in \dom(\DtnNu)$ we have
\begin{equation}
    \DtnNu f=\Pi_{\Snu}(\nor u^f).
\end{equation}

\begin{rem}
\label{rem:dom2}
If $f \in H^1(\Sigma) \cap \Snu$, then \cite[Theorem~4.24(i)]{M00} implies $\nor u^f \in L^2(\Sigma)$, and we conclude that $H^1(\Sigma) \cap \Snu \subseteq \dom(\DtnNu)$. We do not know if the reverse inclusion holds. This amounts to a transmission regularity problem: if the two-sided normal $\nor u^f$ is contained in $L^2(\Sigma)$, does it follow that $f \in H^1(\Sigma)$? See Lemma \ref{lemma:transmission} for a related result.
\end{rem}

\begin{rem}
\label{rem:Lambda}
If each $\chi_j$ is constant, it follows immediately that $\Snu = \S$ and $\DtnNu=\Dtn$. Therefore, in the bipartite case there exists a choice of $\nu$ for which $\DtnNu=\Dtn$; see Lemma \ref{lemma:bi}.
\end{rem}

\begin{rem}
\label{rem:nu}
If $\nu$ and $\tilde\nu$ are two choices of unit normal along $\Sigma$, the resulting Dirichlet-to-Neumann maps are unitarily equivalent, where the unitary transformation on $L^2(\Sigma)$ is multiplication by $(\nu \cdot \tilde\nu)$. In the bipartite case it follows that $\DtnNu$ is unitarily equivalent to $\Dtn$ for any choice of $\nu$.
\end{rem}

\section{The second variation}
\label{sec:var}
We now compute the second variation of $\lambda$, leading to our explicit formula \eqref{Hess1} relating the Hessian to the Dirichlet-to-Neumann map.

We recall that for each $\Omega_j$, $\psi_j$ denotes the $L^2$-normalized ground state and $\nu_j$ is the outward unit normal. Moreover, we let $H_j = \dv \nu_j$ denote the mean curvature of $\pO_j$. Our sign convention (which gives the sphere positive mean curvature) is irrelevant for the following calculation; all that matters is that
\begin{equation}
\label{mean}
	H_i\big|_{\pO_i \cap \pO_j} = - H_j\big|_{\pO_i \cap \pO_j}
\end{equation}
whenever $\Omega_i$ and $\Omega_j$ are neighbors, since $\nu_i = -\nu_j$ on their common boundary.

We start with a simple lemma that allows us to compare a sum of integrals over $\pO_j$ to a single integral over $\Sigma$. The proof is a direct calculation so we leave it out.

\begin{lemma}
\label{intsum}
If $f_j$ is a measurable function on $\pO_j$ for each $j$, then
\begin{equation}
	\sum_{j=1}^k \int_{\pO_j} f_j = \int_\Sigma F,
\end{equation}
where $F\big|_{\pO_i \cap \pO_j} = f_i\big|_{\pO_j} + f_j\big|_{\pO_i}$
for $i \neq j$.
\end{lemma}

The Hessian of $\lambda$ in the $\phi\nu$ direction can be computed as
\begin{equation}
	\Hess \lambda(P)[\phi\nu] = \frac{d^2}{dt^2} \lambda(\varphi_t(P)) \Big|_{t=0},
\end{equation}
where $\varphi_t$ is any one-parameter family of diffeomorphisms of $M$ with $\varphi_0 = \id$, $\varphi_t'\big|_{\Sigma,t=0} = \phi\nu$ and $\varphi_t(P) \in \cEp$ for all $t$. 
(We can not assume that $\varphi_t$ is the one-parameter group generated by the vector field $(E^s \phi)\tilde\nu$ on $M$, as described in Section~\ref{sec:manifold}, since there is no guarantee that this flow will preserve the space of equipartitions.) 

We start by differentiating $\lambda_1(\varphi_t(\Omega_j))$ on the $j$th subdomain. From \cite[eq.~(151)]{G10} we have
\begin{equation}
\label{Grinfeld}
	\frac{d^2}{dt^2} \lambda_1(\varphi_t(\Omega_j)) \Big|_{t=0} = \int_{\pO_j} \left( \left(H_j C_j^2 -C_j' \right)\left(\frac{\p\psi_j}{\p\nu_j}\right)^{\!2} + 2 w_j \frac{\p w_j}{\p \nu_j} \right),
\end{equation}
where $C_j$ and $C_j'$ denote the normal velocity of the flow and its $t$ derivative, evaluated at $t=0$, and $w_j$ is the unique solution to
\begin{equation}
  \label{eq:def_wj}
	\Delta w_j + \lambda_1(\Omega_j) w_j = 0, \qquad w_j\big|_{\pO_j} = - C_j \frac{\p \psi_j}{\p \nu_j}, \qquad
	\int_{\Omega_j} w_j \psi_j = 0.
\end{equation}
The normal velocity at $t=0$  is given by $C_j = (\phi\nu) \cdot \nu_j = \chi_j \phi$. The precise value of the derivative $C'_j$ is irrelevant; it only matters that it is an odd function, in the sense that
\begin{equation}
\label{Cjodd}
	C'_i\big|_{\pO_i \cap \pO_j} = - C'_j\big|_{\pO_i \cap \pO_j}
\end{equation}
whenever $\Omega_i$ and $\Omega_j$ are neighbors. This follows from the observation that the normal velocity is odd for all $t$, since $\nu_i = -\nu_j$ on the common boundary on $\Omega_i$ and $\Omega_j$, and likewise for their deformations $\varphi_t(\Omega_i)$ and $\varphi_t(\Omega_j)$.

By the equipartition condition we have $\lambda(\varphi_t(P)) = \lambda_1(\varphi_t(\Omega_j))$ for each $j$.  For $a_1,\dots,a_k$ as in \eqref{eq:normals}, using our assumption that $a_1^2 + \cdots + a_k^2 = 1$, we can write
\[
    \lambda(\varphi_t(P)) = \sum_{j=1}^k a_j^2 \lambda_1(\varphi_t(\Omega_j)),
\]
and hence
\begin{equation*}
	\frac{d^2}{dt^2} \lambda(\varphi_t(P)) \Big|_{t=0} = \sum_{j=1}^k a_j^2 \frac{d^2}{dt^2} \lambda_1(\varphi_t(\Omega_j)) \Big|_{t=0}.
\end{equation*}
Using \eqref{Grinfeld} to evaluate each term on the right-hand side, we get
\begin{equation}
\label{2var}
	\frac{d^2}{dt^2} \lambda(\varphi_t(P)) \Big|_{t=0} 
	= \sum_{j=1}^k a_j^2 \int_{\pO_j} \left( \left(H_j C_j^2 -C_j' \right)\left(\frac{\p\psi_j}{\p\nu_j}\right)^{\!2} + 2 w_j \frac{\p w_j}{\p \nu_j} \right).
\end{equation}

Next, we use Lemma \ref{intsum} to conclude that
\[
	\sum_{j=1}^k a_j^2 \int_{\pO_j} \left(H_j C_j^2 -C_j' \right)\left(\frac{\p\psi_j}{\p\nu_j}\right)^{\!2} = \int_\Sigma F = 0,
\]
because
\begin{equation}
	F\big|_{\pO_i \cap \pO_j} = \left(H_iC_i^2 - C_i' \right) \left(a_i \frac{\p \psi_i}{\p \nu_i}\right)^{\!2}
	+ \left(H_j C_j^2 - C_j'\right) \left(a_j \frac{\p \psi_j}{\p \nu_j}\right)^{\!2} = 0
\end{equation}
for all $i \neq j$, on account of \eqref{eq:normals}, \eqref{mean} and \eqref{Cjodd}. Substituting this into \eqref{2var} and then integrating by parts, using \eqref{eq:def_wj}, yields
\begin{align*}
	\frac{d^2}{dt^2} \lambda(\varphi_t(P)) \Big|_{t=0} 
	&= 2 \sum_{j=1}^k a_j^2 \int_{\pO_j}  w_j \frac{\p w_j}{\p \nu_j} \\
	&= 2 \sum_{j=1}^k \int_{\Omega_j} \big( |a_j \nabla w_j|^2 - \lambda_* (a_j w_j)^2\big).
\end{align*}

Finally, recalling the definition of $\rho$ in \eqref{rhodef}, we note that
\[
	a_j w_j\big|_{\pO_j} = -\chi_j  a_j \frac{\p \psi_j}{\p \nu_j} \phi = \pm \chi_j \rho \phi,
\]
where the $\pm$ sign is consistent over the entire boundary of $\pO_j$. This means for each $j$ the function $u_j := a_j w_j$ satisfies the boundary value problem
\[
   \Delta u_j + \lambda_* u_j = 0 \ \text{ in } \Omega_j, \qquad\qquad u_j\big|_{\pO_j} = \pm\chi_j \rho\phi,
\]
and so
\[
	\Hess \lambda(P)[\phi\nu] = 2 \sum_{j=1}^k \int_{\Omega_j} \big( |\nabla u_j|^2 - \lambda_* u_j^2\big)
	= 2a(\rho\phi,\rho\phi),
\]
where $a$ is the bilinear form that generates $\DtnNu$, as in \eqref{adef}. This completes the proof of \eqref{Hess1}.

\section{Closing the Hessian}
\label{sec:close}

Having computed the Hessian of $\lambda$, we are now ready to prove our main results.

\begin{proof}[Proof of Theorem \ref{thm:Hess1}]

From \eqref{Hess1} we have
\begin{equation}
\label{Hess2}
	h(\phi_1,\phi_2) = \Hess \lambda(P)(\phi_1\nu,\phi_2\nu) 
	= 2a(\rho\phi_1,\rho\phi_2)
\end{equation}
for all $\phi_1,\phi_2 \in  H^s_\rho(\Sigma) \cap \Fnu$. We then define a form $\bar h(\phi_1,\phi_2) = 2a(\rho\phi_1, \rho\phi_2)$ with 
$$
\dom(\bar h) = H^{1/2}_\rho(\Sigma) \cap \Fnu = \{\phi : \rho\phi \in \dom(a)\}.
$$
It is clear that $\dom(\bar h)$ is dense in $\Fnu$. Using \eqref{a:bound} and \eqref{a:coer}, we conclude that $\bar h$ is closed and semibounded, and hence generates a self-adjoint operator, which we denote $\cH$. Moreover, using the fact that $\phi \in \Fnu$ if and only if $\rho\phi \in \Snu$, we find that
\[
    \dom(\cH) = \big\{\phi : \rho\phi \in \dom(\DtnNu) \big\},
\]
and $\cH\phi = 2 \rho^{-1}\DtnNu(\rho\phi)$. Finally, using the fact that 
\[
    H^1(\Sigma) \cap \Snu \subseteq \dom(\DtnNu) \subseteq  H^{1/2}(\Sigma) \cap \Snu,
\]
we obtain \eqref{eq:domain}, completing the proof.
\end{proof}

Corollary \ref{cor:Morse} follows immediately from Theorem~\ref{thm:Hess1}. To prove Theorem~\ref{thm:smooth}, we will show that the eigenfunctions of $\cH$ are smooth, and hence are contained in the domain of $\Hess \lambda(P)$. The main ingredient in the proof is the following transmission regularity result.

\begin{lemma}\cite[Theorem~4.20]{M00}
\label{lemma:transmission}
Suppose $\Omega_i$ and $\Omega_j$ are neighbors.
If $u_i \in H^1(\Omega_i)$ and $u_j \in H^1(\Omega_j)$ satisfy $\Delta u_i \in H^r(\Omega_i)$, $\Delta u_j \in H^r(\Omega_j)$, 
\[
	u_i\big|_{\pO_i \cap \pO_j} - u_j\big|_{\pO_i \cap \pO_j} \in H^{r+3/2}(\pO_i \cap \pO_j)
\]
and
\[
	\frac{\p u_i}{\p \nu_i} + \frac{\p u_j}{\p \nu_j} \in H^{r+1/2}(\pO_i \cap \pO_j)
\]
for some $r \geq 0$, then $u_i \in H^{r+2}(\Omega_i)$ and $u_j \in H^{r+2}(\Omega_j)$.
\end{lemma}

\begin{proof}[Proof of Theorem \ref{thm:smooth}]
Since \eqref{Morseleq} always holds, we just need to prove the reverse inequality,
\begin{equation}
\label{eq:HH}
	n_-\big(\Hess \lambda(P)\big) \geq n_-(\cH).
\end{equation}
Let $m = n_-(\cH)$, and denote by $\phi_1, \ldots, \phi_m \in \dom(\cH)$ the first $m$ eigenfunctions of $\cH$. To prove \eqref{eq:HH} it suffices to show that 
\[
	\phi_i \in \dom\big(\Hess \lambda(P)\big) = H^s_\rho(\Sigma) \cap \Fnu
\]
for $i = 1, \ldots, m$, since this implies that $\Hess \lambda(P)$ is negative definite on $\operatorname{span}\{\phi_1, \ldots, \phi_m\}$ and hence $n_-\big(\Hess \lambda(P)\big) \geq m$. In fact, we will prove that every eigenfunction of $\cH$ is in $C^\infty(\Sigma)$, and hence is contained in $H^s(\Sigma)$ regardless of the choice of $s$.

Therefore, let $\phi$ be an eigenfunction for $\cH$. It follows from Corollary \ref{cor:Morse} that $f = \rho \phi$ is an eigenfunction for $\DtnNu$. We let $\mu$ denote the corresponding  eigenvalue. Fix $i \neq j$ with $\pO_i \cap \pO_j \neq \varnothing$, and let $c_i,c_j \in \{\pm1\}$ denote the constants $c_i := \chi_i\big|_{\pO_i \cap \pO_j}$ and $c_j := \chi_i\big|_{\pO_i \cap \pO_j}$. We then have functions $u_i \in H^1(\Omega_i)$ and $u_j \in H^1(\Omega_j)$ such that $\Delta u_i + \lambda_* u_i = 0$, $\Delta u_j + \lambda_* u_j = 0$,
\[
    c_i u_i\big|_{\pO_i \cap \pO_j} = c_j u_j\big|_{\pO_i \cap \pO_j} = f,
\]
and
\[
	c_i \frac{\p u_i}{\p \nu_i} + c_j\frac{\p u_j}{\p \nu_j} = \DtnNu f = \mu f \in H^{1/2}(\pO_i \cap \pO_j).
\]
It follows from Lemma \ref{lemma:transmission} with $r=0$ that $c_i u_i \in H^2(\Omega_i)$, and hence $f = c_i u_i\big|_{\pO_i \cap \pO_j} \in H^{3/2}(\pO_i \cap \pO_j)$. This implies
\[
	c_i \frac{\p u_i}{\p \nu_i} + c_j \frac{\p u_j}{\p \nu_j} = \mu f \in H^{3/2}(\pO_i \cap \pO_j),
\]
so we can apply Lemma \ref{lemma:transmission} with $r=1$ to obtain $c_iu_i \in H^3(\Omega_i)$. Proceeding inductively, we find that $f$ is smooth. Since $\rho$ is smooth and nowhere vanishing, it follows that $\phi = \rho^{-1}f$ is smooth, as was to be shown.
\end{proof}

Corollary \ref{cor:bipartite} is now an immediate consequence of Theorem \ref{thm:smooth}, Lemma \ref{lemma:bi} and Remark \ref{rem:Lambda}.

\section{Example: the \texorpdfstring{$(3,1)$}{(3,1)} mode on the square}
\label{sec:example}

We conclude by studying the nodal partition generated by $\psi_*(x,y) = \sin(3\pi x) \sin(\pi y)$ on the unit square, with Dirichlet boundary conditions. We refer to $\psi_*$ as the $(3,1)$ mode, and its nodal set as the $(3,1)$ nodal set. Similarly, the $(1,3)$ mode refers to the eigenfunction $\sin(\pi x) \sin(3\pi y)$ with the same eigenvalue.
While this example does not strictly satisfy the requirements of Theorem~\ref{thm:Hess1}, which for simplicity was only formulated on manifolds without boundary, it can be shown that the theorem remains valid in this case, as will be described in \cite{BCCKM}.

This means we can use \eqref{eigenfunction} to relate eigenfunctions of the two-sided Dirichlet-to-Neumann map $\DtnNu$ to eigenfunctions of the self-adjoint operator $\cH$ generated by $\Hess \lambda(P)$. This is useful because the Dirichlet-to-Neumann eigenfunctions can be computed explicitly in this case, and by taking the eigenfunction corresponding to the most negative eigenvalue, we obtain the direction of steepest descent for the equipartition energy $\lambda$. In Figure \ref{fig:31b} we plot the resulting deformation of the $(3,1)$ nodal partition, and observe that it is moving towards the conjectured minimal 3-partition of the square, which was investigated numerically in \cite{bonnaillie2010numerical}.

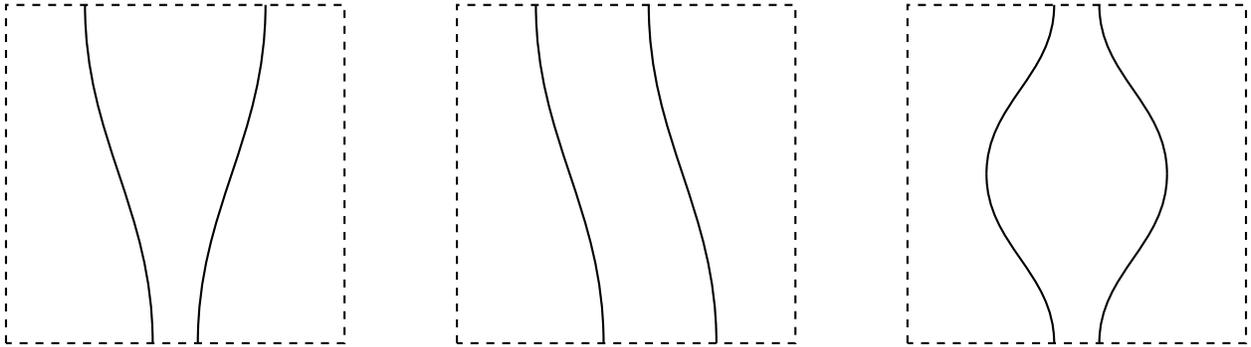
\begin{figure}
\begin{center}
	\begin{tikzpicture}[scale=1.5]
		\draw[thick,dashed] (0,0) -- (0,3) -- (3,3) -- (3,0) -- (0,0);
	\draw[thick] (0.7,3) to [out=270, in=90] (1.3,0);
	\draw[thick] (2.3,3) to [out=270, in=90] (1.7,0);
	\end{tikzpicture}
\hfill
	\begin{tikzpicture}[scale=1.5]
		\draw[thick,dashed] (0,0) -- (0,3) -- (3,3) -- (3,0) -- (0,0);
	\draw[thick] (0.7,3) to [out=270, in=90] (1.3,0);
	\draw[thick] (1.7,3) to [out=270, in=90] (2.3,0);

	\end{tikzpicture}
\hfill
	\begin{tikzpicture}[scale=1.5]
		\draw[thick,dashed] (0,0) -- (0,3) -- (3,3) -- (3,0) -- (0,0);
	\draw[thick] (1.3,3) to [out=270, in=90] (0.7,1.5);
	\draw[thick] (0.7,1.5) to [out=270, in=90] (1.3,0);
\draw[thick] (1.7,3) to [out=270, in=90] (2.3,1.5);
	\draw[thick] (2.3,1.5) to [out=270, in=90] (1.7,0);
	\end{tikzpicture}
\end{center}
\caption{Deformations of the $(3,1)$ nodal set along the Hessian eigenfunctions $\phi_1$, $\phi_2$ and $\phi_3$ (pictured from left to right). The associated eigenvalues are negative for $\phi_1$ and $\phi_2$ and zero for $\phi_3$, which corresponds to deformation along the $(1,3)$ mode.}
\label{fig:31a}
\end{figure}

The nodal set of $\psi_*$ is $\Sigma = \{1/3, 2/3\} \times [0,1]$. We choose $\nu$ so that $\nu\big|_{x=1/3} = (1,0)$ and $\nu\big|_{x= 2/3} = (-1,0)$, hence $\chi_1 = \chi_3 = 1$ and $\chi_2 = -1$. In this case the subspace $\Snu$ defined in \eqref{Schidef} coincides with
\begin{equation}
	\S = \left\{ f \in L^2(\Sigma) : \int_0^1 f\big(\tfrac13,y\big) \sin(\pi y)\, dy 
	= \int_0^1 f\big(\tfrac23,y\big) \sin (\pi y) \, dy = 0 \right\},
\end{equation}
and the weight is $\rho(x,y) =\frac{1}{\sqrt3}  \sin \pi y$.

Separating variables, one finds that $\sigma$ is an eigenvalue of $\Dtn$ if there exists $u(x,y) = g(x)h(y)$ satisfying $\Delta u + \lambda_{3,1} u = 0$ in $\Omega \setminus \Sigma$, with the boundary conditions $g(0) = g(1) = h(0) = h(1) = 0$, the continuity conditions $g\big(\tfrac13+) = g\big(\tfrac13-)$ and $g\big(\tfrac23+) = g\big(\tfrac23-)$, and the jump conditions
\begin{equation}
    g'\big(\tfrac13+) - g'\big(\tfrac13-) = \sigma g\big(\tfrac13), \qquad
    g'\big(\tfrac23+) - g'\big(\tfrac23-) = \sigma g\big(\tfrac23).
\end{equation}
The first two eigenfunctions have $h(y) = \sin(2\pi y)$. It can be shown that the $g(x)$ giving the most negative value of $\sigma$ is even with respect to $x=1/2$, so $g\big(\tfrac13) = g\big(\tfrac23)$, and the corresponding eigenfunction of $\Dtn$, denoted $f_1$, is thus given by
\begin{equation}
    f_1\big(\tfrac13,y\big) = f_1\big(\tfrac23,y\big) = \sin(2\pi y).
\end{equation}
Similarly, the second eigenvalue corresponds to $g(x)$ that is odd with respect to $x=1/2$, hence
\begin{equation}
    f_2\big(\tfrac13,y\big) = -f_2\big(\tfrac23,y\big) = \sin(2\pi y).
\end{equation}
Finally, the third eigenvalue of $\Dtn$, which is zero, has $g(x) = \sin(\pi y)$ and $h(y) = \sin(3\pi y)$, hence
\begin{equation}
    f_3\big(\tfrac13,y\big) =  f_3\big(\tfrac23,y\big) = \sin(3\pi y).
\end{equation}
These formulas for the first three eigenfunctions can also be obtained using the spectral flow method from \cite{BCM19}; we do not elaborate on this here, but refer the reader to \cite{beck2021limiting}, where a similar computation is carried out in detail.

Using \eqref{eigenfunction}, we therefore obtain (up to an overall normalization) the $\mathcal H_{\scriptscriptstyle P}$ eigenfunctions
\begin{align}
	\phi_1\big(\tfrac13,y\big) = \ & \phi_1\big(\tfrac23,y\big) = \frac{\sin(2\pi y)}{\sin(\pi y)}, \\
	\phi_2\big(\tfrac13,y\big) = -&\phi_2\big(\tfrac23,y\big) = \frac{\sin(2\pi y)}{\sin(\pi y)}, \\
	\phi_3\big(\tfrac13,y\big) = \ & \phi_3\big(\tfrac23,y\big) =  \frac{\sin(3\pi y)}{\sin(\pi y)}.
\end{align}
The deformations of the nodal partition $P$ along the vector fields $\phi_1 \nu$, $\phi_2\nu$ and $\phi_3\nu$ are illustrated in Figure \ref{fig:31a}, from left to right.

The appearance of the eigenfunction $\phi_3$ in the kernel of $\mathcal H_{\scriptscriptstyle P}$ is easily understood. For any $t$, $\psi_t(x,y) = \sin(3\pi x) \sin(\pi y) + t \sin(\pi x) \sin(3\pi y)$ is a Laplacian eigenfunction, with eigenvalue $\lambda_{3,1}$ independent of $t$. Letting $P_t$ denote the corresponding nodal partition, we have that $\lambda(P_t)$ is constant in $t$, hence $\Hess \lambda(P_0)(\phi\nu,Y) = 0$ for any normal vector field $Y$ along $\Sigma$, where $\phi\nu$ is the infinitessimal generator of the family $P_t$. Recalling that the normal derivative $\nu \cdot\nabla$ is $\p/\p x$ at $x=1/3$ and $-\p/\p x$ at $x=2/3$, we find that
\begin{equation}
	\phi(x,y) = - \frac{\sin(\pi x) \sin(3\pi y)}{\nu \cdot \nabla(\sin(3\pi x) \sin(\pi y))}\bigg|_{x=1/3,2/3} = \frac{\sqrt3}{6\pi} \frac{\sin (3\pi y)}{\sin (\pi y)},
\end{equation}
is proportional to $\phi_3$, as expected.

On the other hand, the eigenfunction $\phi_1$ corresponds to the most negative eigenvalue of $\Dtn$, and so $\phi_1\nu$ gives the direction of steepest descent for the equipartition energy $\lambda$. The deformation of the nodal partition $P$ along this direction is shown in Figure \ref{fig:31b}. 
The left panel shows the original partition, and the middle panel shows its deformation by $\phi_1\nu$, which pushes apart the nodal lines for $y > \frac12$ and brings them closer together for $y<\frac12$.  The far right panel is an illustration of the conjectured minimal partition, which was computed numerically in \cite{bonnaillie2010numerical}.

These figures suggest that the gradient flow of $\lambda$, with respect to a suitable Riemannian structure on the manifold $\cEp$, will asymptotically approach the conjectured minimum. However, the initial partition and the conjectured minimum have different topology\,---\,the former is smooth and bipartite while the latter is not\,---\,and so the resolution of this problem will require a more detailed study of the space of general (i.e.\ non-generic) equipartitions. This structure will be investigated in a future work \cite{BCCKM}.

\begin{figure}
\begin{center}
	\begin{tikzpicture}[scale=1.5]
		\draw[thick,dashed] (0,0) -- (0,3) -- (3,3) -- (3,0) -- (0,0);
		\draw[thick] (1,0) -- (1,3);
		\draw[thick] (2,0) -- (2,3);
	\end{tikzpicture}
\hfill
	\begin{tikzpicture}[scale=1.5]
		\draw[thick,dashed] (0,0) -- (0,3) -- (3,3) -- (3,0) -- (0,0);
		\draw[thick] (0.7,3) to [out=270, in=90] (1.3,0);
		\draw[thick] (2.3,3) to [out=270, in=90] (1.7,0);
	\end{tikzpicture}
\hfill
	\begin{tikzpicture}[scale=1.5]
		\draw[thick,dashed] (0,0) -- (0,3) -- (3,3) -- (3,0) -- (0,0);
		\draw[thick] (1.5,0) -- (1.5,1.3);
		\draw[thick] (1.5,1.3) to [out=60, in=180] (3,2);
		\draw[thick] (1.5,1.3) to [out=120, in=0] (0,2);
	\end{tikzpicture}
\end{center}
\caption{From left to right: the $(3,1)$ nodal set, its deformation along $\phi_1$ (the direction of steepest descent), and the conjectured minimal 3-partition. }
\label{fig:31b}
\end{figure}
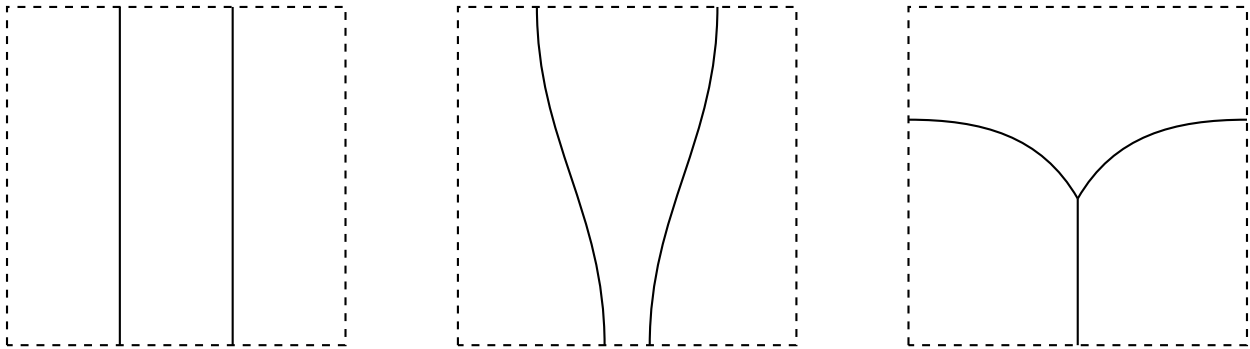

\bibliographystyle{plain}
\bibliography{nodal}

\def\cprime{$'$}
\begin{thebibliography}{10}

\bibitem{AM12}
Wolfgang Arendt and Rafe Mazzeo.
\newblock Friedlander's eigenvalue inequalities and the
  {D}irichlet-to-{N}eumann semigroup.
\newblock {\em Commun. Pure Appl. Anal.}, 11(6):2201--2212, 2012.

\bibitem{BBRS12}
Ram Band, Gregory Berkolaiko, Hillel Raz, and Uzy Smilansky.
\newblock The number of nodal domains on quantum graphs as a stability index of
  graph partitions.
\newblock {\em Comm. Math. Phys.}, 311(3):815--838, 2012.

\bibitem{beck2021limiting}
Thomas Beck, Isabel Bors, Grace Conte, Graham Cox, and Jeremy~L Marzuola.
\newblock Limiting eigenfunctions of {S}turm--{L}iouville operators subject to
  a spectral flow.
\newblock {\em Annales math{\'e}matiques du Qu{\'e}bec}, 45(2):249--269, 2021.

\bibitem{berard2016courant}
Pierre B{\'e}rard and Bernard Helffer.
\newblock Courant-sharp eigenvalues for the equilateral torus, and for the
  equilateral triangle.
\newblock {\em Letters in Mathematical Physics}, 106(12):1729--1789, 2016.

\bibitem{berard2020courant}
Pierre B\'{e}rard, Bernard Helffer, and Rola Kiwan.
\newblock Courant-sharp property for {D}irichlet eigenfunctions on the
  {M}\"{o}bius strip.
\newblock {\em Port. Math.}, 78(1):1--41, 2021.

\bibitem{BCCKM}
Gregory Berkolaiko, Yaiza Canzani, Graham Cox, Peter Kuchment, and Jeremy~L.
  Marzuola.
\newblock Stability of spectral partitions with corners (working title).
\newblock {\em in preparation}.

\bibitem{BCCM}
Gregory Berkolaiko, Yaiza Canzani, Graham Cox, and Jeremy~L. Marzuola.
\newblock A local test for global extrema in the dispersion relation of a
  periodic graph.
\newblock {\em arXiv:2004.12931}, 2021.

\bibitem{BCHS}
Gregory Berkolaiko, Graham Cox, Bernard Helffer, and Mikael~Persson Sundqvist.
\newblock Computing nodal deficiency with a refined {D}irichlet-to-{N}eumann
  map.
\newblock {\em arXiv:2201.06667}, 2022.

\bibitem{BCM19}
Gregory Berkolaiko, Graham Cox, and Jeremy~L. Marzuola.
\newblock Nodal deficiency, spectral flow, and the {D}irichlet-to-{N}eumann
  map.
\newblock {\em Lett. Math. Phys.}, 109(7):1611--1623, 2019.

\bibitem{BKS12}
Gregory Berkolaiko, Peter Kuchment, and Uzy Smilansky.
\newblock Critical partitions and nodal deficiency of billiard eigenfunctions.
\newblock {\em Geom. Funct. Anal.}, 22(6):1517--1540, 2012.

\bibitem{BRS12}
Gregory Berkolaiko, Hillel Raz, and Uzy Smilansky.
\newblock Stability of nodal structures in graph eigenfunctions and its
  relation to the nodal domain count.
\newblock {\em J. Phys. A}, 45(16):165203, 16, 2012.

\bibitem{bonnaillie2015nodal}
Virginie Bonnaillie-No\"{e}l and Bernard Helffer.
\newblock Nodal and spectral minimal partitions---the state of the art in 2016.
\newblock In {\em Shape optimization and spectral theory}, pages 353--397. De
  Gruyter Open, Warsaw, 2017.

\bibitem{bonnaillie2010numerical}
Virginie Bonnaillie-No{\"e}l, Bernard Helffer, and Gregory Vial.
\newblock Numerical simulations for nodal domains and spectral minimal
  partitions.
\newblock {\em ESAIM: Control, Optimisation and Calculus of Variations},
  16(1):221--246, 2010.

\bibitem{CJM2}
Graham Cox, Christopher~K.R.T. Jones, and Jeremy~L. Marzuola.
\newblock Manifold decompositions and indices of {S}chr\"{o}dinger operators.
\newblock {\em Indiana Univ. Math. J.}, 66:1573--1602, 2017.

\bibitem{EM70}
David~G. Ebin and Jerrold Marsden.
\newblock Groups of diffeomorphisms and the motion of an incompressible fluid.
\newblock {\em Ann. of Math. (2)}, 92:102--163, 1970.

\bibitem{G10}
P.~Grinfeld.
\newblock Hadamard's formula inside and out.
\newblock {\em J. Optim. Theory Appl.}, 146(3):654--690, 2010.

\bibitem{helffer2009nodal}
B.~Helffer, T.~Hoffmann-Ostenhof, and S.~Terracini.
\newblock Nodal domains and spectral minimal partitions.
\newblock {\em Ann. Inst. H. Poincar\'{e} Anal. Non Lin\'{e}aire},
  26(1):101--138, 2009.

\bibitem{helffer2010spectral}
Bernard Helffer, Thomas Hoffmann-Ostenhof, and Susanna Terracini.
\newblock On spectral minimal partitions: the case of the sphere.
\newblock In {\em Around the Research of Vladimir Maz'ya III}, pages 153--178.
  Springer, 2010.

\bibitem{helffer2016nodal}
Bernard Helffer and Mikael Sundqvist.
\newblock On nodal domains in {E}uclidean balls.
\newblock {\em Proceedings of the American Mathematical Society},
  144(11):4777--4791, 2016.

\bibitem{helffer2021spectral}
Bernard Helffer and Mikael~Persson Sundqvist.
\newblock Spectral flow for pair compatible equipartitions.
\newblock {\em Communications in Partial Differential Equations}, pages 1--28,
  2021.

\bibitem{HK21}
Matthias Hofmann and James~B. Kennedy.
\newblock Interlacing and {F}riedlander-type inequalities for spectral minimal
  partitions of metric graphs.
\newblock {\em Lett. Math. Phys.}, 111(4):Paper No. 96, 30, 2021.

\bibitem{HKMP21}
Matthias Hofmann, James~B. Kennedy, Delio Mugnolo, and Marvin Pl\"{u}mer.
\newblock Asymptotics and estimates for spectral minimal partitions of metric
  graphs.
\newblock {\em Integral Equations Operator Theory}, 93(3):Paper No. 26, 36,
  2021.

\bibitem{KKLM21}
James~B. Kennedy, Pavel Kurasov, Corentin L\'{e}na, and Delio Mugnolo.
\newblock A theory of spectral partitions of metric graphs.
\newblock {\em Calc. Var. Partial Differential Equations}, 60(2):Paper No. 61,
  63, 2021.

\bibitem{lena2015courant}
Corentin L{\'e}na.
\newblock Courant-sharp eigenvalues of a two-dimensional torus.
\newblock {\em Comptes Rendus Mathematique}, 353(6):535--539, 2015.

\bibitem{M00}
William McLean.
\newblock {\em Strongly elliptic systems and boundary integral equations}.
\newblock Cambridge University Press, Cambridge, 2000.

\bibitem{P56}
{\AA}ke Pleijel.
\newblock Remarks on {C}ourant's nodal line theorem.
\newblock {\em Comm. Pure Appl. Math.}, 9:543--550, 1956.

\bibitem{U76}
K.~Uhlenbeck.
\newblock Generic properties of eigenfunctions.
\newblock {\em Amer. J. Math.}, 98(4):1059--1078, 1976.

\end{thebibliography}

\end{document}